\documentclass{amsart}
\usepackage{amsmath,amsthm,amsopn,amsfonts,fancyhdr,graphicx,epic,amssymb,epsfig,graphics,psfrag}
\newtheorem{thm}{Theorem}[section]

\newtheorem{lemma}[thm]{Lemma}
\newtheorem{prop}[thm]{Proposition}
\theoremstyle{definition}
\newtheorem{defn}[thm]{Definition}
\theoremstyle{remark}
\newtheorem{rem}[thm]{Remark}
\numberwithin{equation}{section}

\newcommand{\norma}[1]{\|#1\|}
\newcommand{\thetabar}{\bar{\theta}}

\newcommand{\N}{\mathbb{N}}
\newcommand{\R}{\mathbb{R}}

\def\build#1_#2^#3{\mathrel{\mathop{\kern 0pt#1}\limits_{#2}^{#3}}}

\newcommand{\normasemplice}[1]{\|#1\|}

\newcommand{\dys}{\displaystyle}
\newcommand{\matrice}[4]
{\left(\begin{array}{cc} #1&#2
\\
#3&#4
\end{array} \right) }
\newcommand{\matricesimmetrica}[3]
{\left(\begin{array}{cc} #1&#2
\\
#2&#3
\end{array} \right) }

\newcommand{\rife}[1]{(\ref{#1})}
\begin{document}

\title[A differential inclusion: the case of an isotropic set]
{A differential inclusion: the case of an isotropic set}
\author{Gisella Croce}
\address{D\'epartement de Math\'ematiques\\EPFL\\1015 Lausanne}
\email{gisella.croce@epfl.ch}%

\thanks{Research supported by Fonds National Suisse (21-61390-00).}


\begin{abstract}
In this article we are interested in the following problem: to
find a map $u: \Omega \to \R^2$ that satisfies
\[
\left\{
\begin{array}{ll}
D u \in E\,\, &\mbox{a.e. in } \Omega
\vspace{0.1cm}\\
u(x)=\varphi(x) &x \in \partial \Omega \vspace{0.1cm}
\end{array}
\right.
\]
where $\Omega$ is an open set of $\R^2$ and $E$ is a
compact isotropic set of $\R^{2\times 2}$. We will show an
existence theorem under suitable hypotheses on $\varphi$.
\end{abstract}
\maketitle
\section{Introduction}
In this article we study the following problem: let $\Omega$ be an
open set of $\R^2$; we investigate the existence of maps $u:
\Omega \to \R^2$ (weakly differentiable) that satisfy
\begin{equation}
\left\{
\begin{array}{ll}
D u \in E\,\, &\mbox{a.e. in } \Omega
\vspace{0.1cm}\\
u(x)=\varphi(x) &x \in \partial \Omega \vspace{0.1cm}
\end{array}
\right. \label{problema di partenza}
\end{equation}
where $\varphi:\overline{\Omega}\to \R^2$ is sufficiently regular
and $E$ is a compact isotropic set of $\R^{2\times 2}$ (that is
$AEB\subseteq E$ for every $A,B\in \mathcal{O}(2)$). In an
equivalent way $E$ can be written as
\begin{equation}\label{notreE}
E=\{\xi \in \R^{2\times 2}: (\lambda_1(\xi), \lambda_2(\xi))\in K
\}\,,
\end{equation}
for some compact set $K\subset T=\{(x,y)\in \R^{2}: 0\leq x\leq
y\}$, where we have denoted by $\lambda_1(\xi)\leq \lambda_2(\xi)$
the singular values of the matrix $\xi$, that is the eigenvalues
of the matrix $\sqrt{\xi \xi^t}$, which are
\begin{equation}
\begin{array}{l}
\displaystyle \lambda_1(\xi)=\frac 12
\left[\sqrt{\normasemplice{\xi}^2+2
|\det(\xi)|}-\sqrt{\normasemplice{\xi}^2-2 |\det(\xi)|}\right]
\vspace{0.15cm}
\\
\displaystyle \lambda_2(\xi)=\frac 12
\left[\sqrt{\normasemplice{\xi}^2+2
|\det(\xi)|}+\sqrt{\normasemplice{\xi}^2-2 |\det(\xi)|}\right]\,.
\end{array}
\label{valorisingolari}
\end{equation}
We will assume throughout the article that
\begin{equation}
\min\{x, (x,y)\in K\}>0\,.
\label{ipotesidipositivitadeipuntidiK}
\end{equation}
 Thanks to the
properties of the singular values (see \cite{libromatrici}), the
problem (\ref{problema di partenza}) can be rewritten in the
following equivalent way:
\[
\left\{
\begin{array}{ll}
\norma{D u(x)}^2=a^2+b^2\,\, &\mbox{a.e. in } \Omega, (a,b)\in K,
\vspace{0.1cm}\\
|\text{det}\,D u(x)|=ab\,\, &\mbox{a.e. in } \Omega, (a,b)\in K,
\\
u(x)=\varphi(x) &x \in \partial \Omega\,.\end{array} \right.
\]
We will show the following existence theorem for the problem
(\ref{problema di partenza}):
\begin{thm}\label{teorema d'esistenza nel nostro caso dacorognamarcellini}
Let $E$ be defined by \rife{notreE} where $K\subset T$ is a
compact set such that $\min\{x, (x,y)\in K\}>0\,.$  Let
$\Omega\subset \R^2$ be an open set. Let $\varphi\in
C^1_{piec}(\overline{\Omega};\R^2)$ such that $D \varphi(x) \in
E\cup \textnormal{int Rco}E$ in $\Omega$ (where
$\textnormal{Rco}\,E$ denotes the rank one convex hull of $E$ and
$\textnormal{int Rco}E$ its interior). Then there exists a map $u
\in \varphi+W^{1,\infty}_0(\Omega; \R^2)$ such that $D u \in E$
a.e. in $\Omega$.
\end{thm}
Our result will be a generalization of an existence theorem by
Dacorogna and Marcelli\-ni (see \cite{implicitdacomarc}), where
they investigated the case in which $K$ contains only one point,
but $K\subseteq \R^n,\, n\geq 2$.

To establish our theorem we will use an abstract existence result
proved by Dacorogna and Marcellini \cite{implicitdacomarc} and
then refined by Dacorogna and Pisante \cite{Giovio}, which is
based on a functional analytic method, that uses the Baire
category theorem. However we recall that the kind of problem like
(\ref{problema di partenza}) can also be solved by another method,
established by Gromov \cite{gromovbook} and then presented by
M\"{u}ller and \v{S}ver\'ak (see \cite{gromovintegration} for
example) in a more analytic manner.

Moreover it will be useful the representation of the rank one
convex hull of the set $E$ (we will show that the rank one convex
hull coincides with the polyconvex one): for this, we will apply
some results established by Cardaliaguet and Tahraoui in
\cite{card-tah}.

\section{Definitions and preliminary results}
This section is devoted to the study of the polyconvex hull of an
isotropic compact set of $\R^{2\times 2}$: this is useful to study
the rank one convex hull, as we will see in the next section.

We first give the definitions of polyconvex hull and rank one
convex hull of a set $E\subseteq\R^{2\times 2}\,:$ we will follow
the definitions of Dacorogna and Marcellini
\cite{implicitdacomarc}.
\begin{defn} Let $E\subseteq\R^{2\times 2};$  if
$\overline{\R}=\R\cup \{+\infty\}$, we define
\[
\begin{array}{lcl}\textnormal{Pco}\,E &= & \{\xi \in
\R^{2\times 2}: f(\xi)\leq 0, \forall\,\, f:\R^{2\times 2}\to
\overline{\R} \,\,\mbox{ polyconvex}\,\,, f|_{E}\leq 0 \} ;
\\\textnormal{Rco}\,E &=&\{\xi \in \R^{2\times 2}: f(\xi)\leq 0,
\forall\,\, f:\R^{2\times 2}\to \overline{\R} \,\,\mbox{rank one
convex}\,\,, f|_{E}\leq 0 \}
 .
\end{array}
\]
\end{defn}
\begin{rem}\label{rcoe}
{\it {i)}} We observe that, according to this definition, the rank
one convex hull of a compact set is not necessarily closed (some
examples can be found in \cite{kolar}).  According to our
definition $\textnormal{Rco}\,E$ is the smallest rank one convex
set which contains $E$ (see \cite{implicitdacomarc}). Some authors
call our envelop the laminate convex hull of $E$.

\noindent {\it {ii)}} We will use also the following
representation for {\textnormal{Rco}\,E} (see
\cite{implicitdacomarc}): $\textnormal{Rco}E=\bigcup\limits_{i \in
\N} \textnormal{R}_i \textnormal{co}E,$ where
$$\textnormal{R}_0 \textnormal{co}E=E, \quad \textnormal{R}_{i+1}
\textnormal{co}E=\{tA+(1-t)B: A,B \in \textnormal{R}_i
\textnormal{co}E, rk(A-B)\leq 1, t\in [0,1]\}\,.
$$

\noindent {\it {iii)}} If $E$ is bounded, then (see
\cite{implicitdacomarc})
$$
 \overline{\textnormal{Pco}\,E} = \{\xi \in
\R^{2\times 2}: f(\xi)\leq 0, \forall\,\, f:\R^{2\times 2}\to {\R}
\,\,\mbox{ polyconvex}\,\,, f|_{E}\leq 0 \}=
\textnormal{Pco}\,\overline{E}\,.
$$
This means that for a compact set, our notion of polyconvexity
coincides with that one of Cardaliaguet and Tahraoui in
\cite{card-tah}. We observe that the polyconvex hull of a compact
set $E$ is compact.

 \noindent {\it {iv)}} We recall that
${\textnormal{Rco}\,E}\subseteq {\textnormal{Pco}\,E}$ for every
set $E$.

\end{rem}
We now pass to the study of the polyconvex hull of our set $E$. We
recall the following result established in \cite{card-tah}:
\begin{prop}\label{policard}Let $K\subset T$  be a compact set and
$E=\{\xi \in \R^{2\times 2}: (\lambda_1(\xi),\lambda_2(\xi))\in
K\}\,.$ Let
\begin{equation}
\begin{array}{l}
\dys
\Sigma=\left\{(\theta,\gamma)\in \R^2_+: \gamma\geq
\theta^2\,\,\textnormal{et}\,\, \forall\,\,(x,y)\in K\,\, y\leq
\theta+\frac{\gamma-\theta^2}{x+\theta}\right\} \vspace{0.2cm}
\\\dys
\sigma(x)=\inf\limits_{(\theta,\gamma)\in
\Sigma}\theta+\frac{\gamma-\theta^2}{x+\theta},\,\, \forall
\,\,x\geq 0.
\end{array}
\label{sigmacard-ta}
\end{equation}
Then
$$
\textnormal{Pco}{E}=\{\xi \in \R^{2\times 2}:
\lambda_2(\xi)\leq\sigma(\lambda_1(\xi))\}\,.
$$
\end{prop}
It will be useful for our purposes to write in a different way the
function $\sigma$ defined by \rife{sigmacard-ta}. We will use the
notation $m(\theta)=\max\limits_ {(a,b) \in K}
\{ab+\theta(b-a)\}\,$ throughout all this article.
\begin{prop}\label{prelim}
Let $K\subset T$ be a compact set satisfying
\rife{ipotesidipositivitadeipuntidiK} and $\sigma$ be the function
defined by \rife{sigmacard-ta}. Then
\begin{equation}
\sigma(x)= \min\limits_{\theta \in [0,\max\limits_ {(a,b) \in K}
b]} \frac{\theta x+m(\theta)}{x+\theta}. \label{sigmagise}
\end{equation}
\end{prop}
\begin{proof}
We divide the proof into two steps: in the first one we study the
set $\Sigma$ defined in proposition \ref{policard}, and in the
second one we state the formula for the function $\sigma$.

\noindent {\it{Step 1: Study of the set $\Sigma$:}} By definition
$$
\Sigma=\left\{(\theta,\gamma)\in \R^2_+: \gamma\geq
\theta^2\,\,\textnormal{et}\,\, \forall\,\,(a,b)\in K\,\, b\leq
\theta+\frac{\gamma-\theta^2}{a+\theta}\right\}\,.
$$
As for every $(a,b) \in K$ we have $a>0\,,$ thanks to
(\ref{ipotesidipositivitadeipuntidiK})
$$\Sigma=\left\{(\theta,\gamma)\in \R^2_+:
\gamma\geq\sup\{\theta^2,\max\limits_{(a,b) \in
K}\{ab+\theta(b-a)\}\}\right\} \,.
$$
We observe that if $\theta\geq \max\limits_{(a,b)\in K}b$ then for
every $(a,b)\in K$ one has
$$
ab+\theta(b-a)\leq \theta a+\theta(b-a)=\theta b\leq \theta^2
$$
and so $\max\limits_{(a,b) \in K}\{ab+\theta(b-a)\}\leq \theta^2$.

\noindent If $\theta< \max\limits_{(a,b)\in K}b$ then, considering
$(a,\max\limits_{(a,b)\in K}b) \in K$ we have
$$
\max\limits_{(a,b) \in K}\{ab+\theta(b-a)\}\geq
a(\max\limits_{(a,b)\in K}b-\theta)+\theta \max\limits_{(a,b)\in
K}b> \theta \max\limits_{(a,b)\in K}b> \theta^2\,.
$$
If $\theta=\max\limits_{(a,b)\in K}b$ we have
\begin{equation}\max\limits_{(a,b) \in K}\{ab+[\max\limits_{(a,b)\in
K}b](b-a)\}=[\max\limits_{(a,b)\in K}b]^2. \label{uguale}
\end{equation}

From this study we can infer that
$$\Sigma=\Sigma_1\cup \Sigma_2,$$ where
\[
\begin{array}{l}
\Sigma_1=\{(\theta,\gamma)\in \R^2: \theta \in
[0,\max\limits_{(a,b)\in K}b], \,\,\gamma\geq
\max\limits_{(a,b)\in K} \{ab+\theta(b-a)\}\}
\\
\Sigma_2=\{(\theta,\gamma)\in \R^2: \theta \geq
\max\limits_{(a,b)\in K}b, \,\,\gamma\geq \theta^2 \} \,.
\end{array}
\]

 \noindent
{\it {Step 2: Study of the function $\sigma$:}} We define
$$
g_x(\theta,\gamma)=\frac{\theta x+
\gamma}{x+\theta}\,\,\,\text{for}\,\,x\geq 0,
$$ and
\[
\begin{array}{l}
\Gamma_1=\{(\theta,\max\limits_{(a,b)\in K}\{ab+\theta(b-a)\}),
\theta \in [0,\max\limits_{(a,b)\in K}b]\} \\
\Gamma_2=\{(\theta,\theta^2), \theta\geq \max\limits_{(a,b)\in
K}b\}\,.
\end{array}
\]
We observe that $\Gamma_1\cup \Gamma_2\subset \Sigma_1\cup
\Sigma_2\,.$ We are going to show that
$$\sigma(x)=\inf\limits_{\Gamma_1\cup
\Gamma_2}g_x(\theta,\gamma).$$ We know from proposition
\ref{policard} that $\dys\sigma(x)=\inf\limits_{(\theta,\gamma)\in
\Sigma}\theta+\frac{\gamma-\theta^2}{x+\theta}$. Now, if $x=0$
$$
\sigma(0)=\inf\limits_{(\theta,\gamma)\in
\Sigma}\theta+\frac{\gamma-\theta^2}{\theta}=\inf\limits_{(\theta,\gamma)\in
\Sigma}\frac{\gamma}{\theta}\,.
$$
 As ($\theta_0$ fixed)
the function $\dys\frac{\gamma}{\theta_0}$ is increasing,
$$
\sigma(0)=\inf\limits_{(\theta,\gamma)\in
\Sigma}\frac{\gamma}{\theta}=\inf\limits_{\Gamma_1\cup
\Gamma_2}g_0(\theta,\gamma).
$$
If $x>0$, $\dys\sigma(x)=\inf\limits_{(\theta,\gamma)\in
\Sigma}\theta+\frac{\gamma-\theta^2}{x+\theta}=\dys\inf\limits_{(\theta,\gamma)\in
\Sigma}g_x(\theta,\gamma)\,.$ We observe that $\dys\frac{\partial
g_x}{\partial \gamma}>0$: this implies that, if $\theta_0\geq 0$
is fixed, $g_x(\theta_0,\gamma)$ is increasing (in $\gamma$), and
so for $x>0$
$$
\sigma(x)=\inf\limits_{\Sigma}g_x(\theta,\gamma)=
\inf\limits_{\Gamma_1\cup\Gamma_2}g_x(\theta,\gamma)\,,
$$
as wished. We observe that
$$g_x\mid_{\Gamma_1}=\frac{\theta x+\max\limits_{(a,b) \in K}\{ab+\theta
(b-a)\}}{x+\theta},
$$
and so
$$
\inf\limits_{\Gamma_1}g_x=\inf\limits_{\theta \in [0,\max\limits_
{(a,b) \in K} b]} \frac{\theta x+\max\limits_{(a,b) \in
K}\{ab+\theta (b-a)\}}{x+\theta}\,.$$ Moreover
$$g_x\mid_{\Gamma_2}=\frac{\theta x+\theta^2}{x+\theta}=\theta,\,\,
\theta\geq \max\limits_{(a,b)\in K}b\,:
$$
consequently $$ \inf\limits_{\Gamma_2}g_x=\max\limits_{(a,b)\in
K}b.$$ This study implies that
$$\sigma(x)= \inf\limits_{\Gamma_1\cup\Gamma_2}g_x(\theta,\gamma)
=\inf\left\{\inf\limits_{\theta \in [0,\max\limits_ {(a,b) \in K}
b]} \frac{\theta x+\max\limits_{(a,b) \in K}\{ab+\theta
(b-a)\}}{x+\theta}, \max\limits_{(a,b)\in K}b\right\}\,.
$$
Due to (\ref{uguale})
$$\frac{ x\max\limits_{(a,b)\in K}b +\max\limits_{(a,b) \in K}
\{ab+[\max\limits_{(a,b)\in K}b] (b-a)\}}{x+\max\limits_{(a,b)\in
K}b}=\max\limits_{(a,b)\in K}b;$$ therefore one has
$$\sigma(x)= \inf\limits_{\theta \in [0,\max\limits_
{(a,b) \in K} b]} \frac{\theta x+\max\limits_{(a,b) \in
K}\{ab+\theta (b-a)\}}{x+\theta}=\inf\limits_{\theta \in
[0,\max\limits_ {(a,b) \in K} b]} \frac{\theta
x+m(\theta)}{x+\theta}.
$$
We are going to show that $\forall\,\,x\geq 0$
$$
\inf\limits_{\theta \in [0,\max\limits_ {(a,b) \in K} b]}
\frac{\theta x+m(\theta)}{x+\theta}=\min\limits_{\theta \in
[0,\max\limits_ {(a,b) \in K} b]} \frac{\theta
x+m(\theta)}{x+\theta}\,,
$$
and so the formula of the statement. For $x>0$ this is trivial.
For $x=0$ we have to study the function of $\theta$
$$
\frac{m(\theta)}{\theta}=\frac{\max\limits_{(a,b) \in
K}\{ab+\theta (b-a)\}}{\theta}:
$$
we observe that,
$$
\frac{\max\limits_{(a,b) \in K}\{ab+\theta(b-a)\}}{\theta}\geq
\frac{\max\limits_{(a,b) \in K}{ab}}{\theta}\to \infty, \theta\to
0.
$$
This implies that there exists  $\bar{\varepsilon}>0$ such that
\begin{equation}
\inf\limits_{\theta \in [0,\max\limits_ {(a,b) \in K}
b]}\frac{m(\theta)}{\theta}=\inf\limits_{\theta \in
[\bar{\varepsilon},\max\limits_ {(a,b) \in K}
b]}\frac{m(\theta)}{\theta}=\min\limits_{\theta \in
[\bar{\varepsilon},\max\limits_ {(a,b) \in K}
b]}\frac{m(\theta)}{\theta}\,, \label{thetadiversodazero}
\end{equation}
and so we have the result.
 \end{proof}
In the next proposition we study the interior of
$\textnormal{Pco}\,E\,.$
\begin{prop}\label{interno}
Let $K\subset T$ be a compact set satisfying
\rife{ipotesidipositivitadeipuntidiK}. Let $ E=\{\xi \in
\R^{2\times 2}: (\lambda_1(\xi), \lambda_2(\xi))\in K \}\,. $
Then, if $\sigma$ is the function defined by (\ref{sigmagise}),
$$
\textnormal{int}\, {\textnormal{Pco}E}=\{\xi \in \R^{2\times 2}:
\lambda_2(\xi)<\sigma(\lambda_1(\xi))\}\,.
$$
\end{prop}
\begin{proof} The proof is divided into two steps.

\noindent {\it{Step 1)}} To show that $\textnormal{int}\,
{\textnormal{Pco}E}\supseteq\{\xi \in \R^{2\times 2}:
\lambda_2(\xi)<\sigma(\lambda_1(\xi))\}\,,$ as $\lambda_i(\xi)$
are continuous functions, $i=1,2$, it is sufficient to show that
the function $\sigma$ is continuous for $x\geq 0$. This is easy
for $x>0$. To study the point $x=0$, we are going to show that if
$x_n\to 0^+$, then $\sigma(x_n)\to \sigma(0)$. We can say from
(\ref{thetadiversodazero}) that there exists $\theta_0\in
(0,\max\limits_{(a,b)\in K}b]$ such that
$\sigma(0)=\frac{m(\theta_0)}{\theta_0}\,.$ Then, by definition of
$\sigma(x_n)$
$$
\begin{array}{l}
\dys\sigma(x_n)-\sigma(0)\leq \dys\frac{\theta_0
x_n+m(\theta_0)}{x_n+\theta_0}-\frac{m(\theta_0)}{\theta_0}\to 0,
n\to \infty\,.
\end{array}
$$
We are going to show that $\sigma(x_n)-\sigma(0)\geq h(x_n),$ with
$h(x_n)\to 0$, if $n\to \infty$: this will imply that
$\sigma(x_n)\to \sigma(0)\,,$ and so the continuity of $\sigma$ in
$0$.

For every $x_n$ there exists $\theta_n \in
[0,\max\limits_{(a,b)\in K}b]$ such that
$$
\sigma(x_n)=\frac{\theta_n x_n+m(\theta_n)}{x_n+\theta_n}\,.
$$Then, by definition of
$\sigma(0)$
$$
\begin{array}{l}
\dys\sigma(x_n)-\sigma(0)=\frac{\theta_n
x_n+m(\theta_n)}{x_n+\theta_n} -\min\limits_{\theta \in
[0,\max\limits_ {(a,b) \in K} b]} \frac{ m(\theta)}{\theta}
\vspace{0.2cm}\\
\phantom{\dys\sigma(x_n)-\sigma(0)}\geq \dys\frac{\theta_n
x_n+m(\theta_n)}{x_n+\theta_n}-\frac{m(\theta_n)}{\theta_n}=
\frac{x_n}{\theta_n(x_n+\theta_n)}[\theta_n^2-m(\theta_n)]\,.
\end{array}
$$
We will show that
$h(x_n)=\frac{x_n}{\theta_n(x_n+\theta_n)}[\theta_n^2-m(\theta_n)]\to
0, n\to \infty\,.$ As $\theta_n^2-m(\theta_n)$ is bounded, it is
now sufficient to prove that
$$
\frac{x_n}{\theta_n(x_n+\theta_n)}\to 0, n\to \infty\,.
$$
We observe that $\liminf\limits_{n\to \infty}\theta_n>0\,.$ In
fact, if $\liminf\limits_{n\to \infty}\theta_n=0\,,$ then there
exists a sub-sequence ${n_k}$ such that $\lim\limits_{k\to
\infty}\theta_{n_k}=0$; consequently
$$
\sigma(x_{n_k})=\frac{\theta_{n_k}
x_{n_k}+m(\theta_{n_k})}{x_{n_k}+\theta_{n_k}}\geq
\frac{\max\limits_{(a,b)\in K}ab}{x_{n_k}+\theta_{n_k}}\to \infty,
k\to \infty\,.
$$
The matrix
$$
\xi_{n_k}=\matricesimmetrica{x_{n_k}}{0}{\sigma(x_{n_k})}
$$ belongs to $\textnormal{Pco}\,E$ and
 $\lambda_2(\xi_{n_k})\to \infty\,:$ recalling
 that $\lambda_2$ is a norm over
 $\R^{2\times 2}$ (see \cite{libromatrici}),
we got  a contradiction because $\textnormal{Pco}E$ is bounded, as
$E$ is bounded.

Then $\dys\liminf\limits_{n\to \infty}\theta_n=a>0\,:$ this
implies that, for $n$ sufficiently large $\theta_n\geq \frac a2$
and so
$$
\frac{x_n}{\theta_n(x_n+\theta_n)}\leq \frac{x_n}{\frac
a2(x_n+\frac a2)} \to 0, n\to \infty\,,
$$
that is the result.

 \noindent {\it{Step 2)}} We now show that
$ \textnormal{int}\,{\textnormal{Pco}}\,{E}\subseteq\{\xi \in
\R^{2\times 2}: \lambda_2(\xi)<\sigma(\lambda_1(\xi))\}\,. $
Suppose that there exists a matrix $\eta \in
\textnormal{int}\,{\textnormal{Pco}}\,{E}$ such that
$\lambda_2(\eta)=\sigma(\lambda_1(\eta))$; therefore the ball
$B_{\varepsilon}(\eta)\subseteq \textnormal{Pco}\,{E}$, for some
$\varepsilon>0$. Let $A,B \in \mathcal{O}(2)$ be such that
$$
A\eta B=\matricesimmetrica{\lambda_1(\eta)}{0}{\lambda_2(\eta)};
$$
we define
$$
D=A^{-1}\matricesimmetrica{0}{0}{d}B^{-1},
$$
with $0<d<\varepsilon$. Then we have
\[
\begin{array}{l}
\lambda_1(\eta+D)=\lambda_1(A\eta B+ADB)=\lambda_1(\eta)
\\
\lambda_2(\eta+D)=\lambda_2(A\eta B+ADB)=\lambda_2(\eta)+d\,.
\end{array}
\]
The matrix $\eta+D \in B_{\varepsilon}(\eta)\subseteq
{\textnormal{Pco}}\,{E}$, as $d<\varepsilon$: this implies that
$\lambda_2(\eta+D)\leq \sigma(\lambda_1(\eta+D)).$ From other
hand,
$$
\lambda_2(\eta+D)=\lambda_2(\eta)+d=\sigma(\lambda_1(\eta))+d=
\sigma(\lambda_1(\eta+D))+d>\sigma(\lambda_1(\eta+D))\,,
$$
and this is a contradiction: therefore
$\lambda_2(\eta)<\sigma(\lambda_1(\eta))\,.$
\end{proof}
\begin{rem}\label{bord}The previous results imply that if
$\xi\in
\partial \textnormal{Pco}\,E,$ then there exists $\thetabar \in
[0,\max\limits_{(a,b)\in K}b]$ such that
$$
\lambda_1(\xi)\lambda_2(\xi)+\overline{\theta}(\lambda_2(\xi)-\lambda_1(\xi))
=m(\thetabar)\,.
$$
\end{rem}
\section{The rank one convex hull}
In this section we are going to prove the representation theorem
of the rank one convex hull of $E$. It will be useful to show our
existence result. We recall that we use the notation
$m(\theta)=\max\limits_{(a,b)\in K}\{ab+\theta(b-a)\}\,.$
\begin{thm}\label{rcoegise}
Let $K\subset T$ be a compact set satisfying
\rife{ipotesidipositivitadeipuntidiK}. Let
$$
E=\{\xi \in \R^{2\times 2}: (\lambda_1(\xi), \lambda_2(\xi))\in K
\}\,.
$$
Then
\[
\begin{array}{cll}
\textnormal{Rco}E&\!\!\!\!\!=&\!\!\!\!\!\left\{\xi \in \R^{2\times
2}:
\lambda_1(\xi)\lambda_2(\xi)+\theta(\lambda_2(\xi)-\lambda_1(\xi))\leq
m(\theta), \forall\,\,\theta \in [0,\max\limits_{(a,b)\in
K}b]\right\}, \vspace{0.2cm}
\\
\textnormal{int\,Rco}E&\!\!\!\!\!=&\!\!\!\!\!\left\{\xi \in
\R^{2\times 2}:
\lambda_1(\xi)\lambda_2(\xi)+\theta(\lambda_2(\xi)-\lambda_1(\xi))<
m(\theta), \forall\,\,\theta \in [0,\max\limits_{(a,b)\in
K}b]\right\}.
\end{array}
\]
\end{thm}
 We will first prove the
following lemma:
\begin{lemma}\label{numerofinitodipunti}
Let $K=\{(a_1,b_1), (a_2,b_2)\}$, $0<a_1<a_2, a_1b_1\leq a_2b_2,
 b_2\leq b_1, a_1 \leq b_1, a_2 \leq b_2$ and $E=\{\xi \in \R^{2\times 2}:
 (\lambda_1(\xi),\lambda_2(\xi))\in K\}$.
 Then
\[
\begin{array}{l}
\textnormal{Rco}\,E=\{\xi \in \R^{2\times 2}: \lambda_2(\xi)\leq
b_1,
\\
\phantom{\textnormal{Rco}\,E=\{\xi \in \R^{2\times
2}:}\lambda_1(\xi)\lambda_2(\xi)\leq a_2 b_2
\\
\phantom{\textnormal{Rco}\,E=\{\xi \in \R^{2\times 2}:}
\lambda_1(\xi)\lambda_2(\xi)+\bar{\theta}(\lambda_2(\xi)-\lambda_1(\xi))\leq
a_1b_1+\bar{\theta}(b_1-a_1) \}
\end{array}
\]
where $\dys\bar{\theta}=\frac{a_2b_2-a_1b_1}{b_1-a_1-b_2+a_2}\,.$
\end{lemma}
\begin{rem}\label{onepoint}
We remember that Dacorogna and Marcellini \cite{implicitdacomarc}
proved that if $K$ is composed by one point $(a,b)$ we have
$$
\textnormal{Rco}\,E=\{\xi \in \R^{2\times 2}: \lambda_2(\xi)\leq
b, \lambda_1(\xi)\lambda_2(\xi)\leq ab\}\,.
$$
\end{rem}
\begin{proof}In \cite{card-tah}
it is showed that the function $\sigma$ (\ref{sigmagise}) defined
for $E$ is
$$\sigma(x)=\inf\left\{b_1, \frac{a_2b_2}{x}, \frac{\thetabar
x+a_1b_1+\thetabar (b_1-a_1) }{\thetabar+x}\right\}\,.
$$
 Thanks to
proposition \ref{policard}, this implies that $\text{Pco}\,E$ is
the set of matrices $\xi$ such that
\begin{eqnarray}
& & \lambda_2(\xi)\leq b_1   \label{primarcoe2},
\\
& &\lambda_1(\xi)\lambda_2(\xi)\leq a_2 b_2\,,
\label{secondarcoe2}
\\
&
&\lambda_1(\xi)\lambda_2(\xi)+\bar{\theta}(\lambda_2(\xi)-\lambda_1(\xi))\leq
a_1b_1+\bar{\theta}(b_1-a_1)\,.  \label{terzarcoe2}
\end{eqnarray}
Therefore to prove the formula of Rco$E$ it is sufficient to prove
that ${\textnormal{Pco}E}=\textnormal{Rco}E.$ For this we will
show that
$$\partial{\textnormal{Pco}E}\subseteq
\textnormal{Rco}E:
$$this will imply the
not trivial inclusion ${\textnormal{Pco}E}\subseteq
\textnormal{Rco}E$ and so the result. In fact, let $\xi \in
\textnormal{int}{\textnormal{Pco}E}$; as ${\textnormal{Pco}E}$ is
compact, then for any rank one matrix $\lambda \in \R^{2\times
2}$, there exist $t_1=t_1(\lambda)<0<t_2=t_2(\lambda)$ such that
$\xi+t_i\lambda\in
\partial{\textnormal{Pco}E}\subseteq \textnormal{Rco}E, i=1,2$.
Defining $\xi_{i}=\xi+t_i\lambda, i=1,2$, we have
$$
\xi=\frac{t_2}{t_2-t_1}\xi_1-\frac{t_1}{t_2-t_1}\xi_2\in
\textnormal{Rco}E\,,
$$
because $\textnormal{rk}(\xi_1-\xi_2)=1.$

Now, let $\xi \in
\partial{\textnormal{Pco}E}$: necessarily $\lambda_2(\xi)=\sigma(\lambda_1(\xi))$
and so $\xi$ satisfies either (\ref{primarcoe2}) or
(\ref{secondarcoe2}) or (\ref{terzarcoe2}) with equality. We are
going to treat these cases separately (steps 1, 2, 3 respectively)
to show that $\xi \in \textnormal{Rco}E$. We can assume without
loss of generality that
$\xi=\textnormal{diag}(\lambda_1(\xi),\lambda_2(\xi))$, as
$\text{Rco}\,E$ is isotropic: in fact, using the same notations as
in the second point of remark \ref{rcoe}, we have by induction on
$i$ that $\textnormal{R}_i \textnormal{co}E$ is isotropic and so
$\textnormal{Rco}E$ is isotropic.

{\it{step 1)}} If $\xi$ satisfies (\ref{primarcoe2}) with
equality, (\ref{terzarcoe2}) implies that $\lambda_1(\xi)\leq
a_1$; then $\xi \in \textnormal{Rco}E,$ as
$$
\xi=\matricesimmetrica{\lambda_1(\xi)}{0}{b_1}=
t\matricesimmetrica{a_1}{0}{b_1}+(1-t)\matricesimmetrica{-a_1}{0}{b_1},
t \in (0,1)\,.
$$
In the next steps we can assume that $\xi$ satisfies
(\ref{primarcoe2}) with strict inequality.

{\it{step 2)}} We suppose that $\xi$ satisfies
(\ref{secondarcoe2}) with equality. Moreover we can assume that
$\xi$ satisfies (\ref{terzarcoe2}) with strict inequality,
otherwise these two equalities imply
$$\xi=\matricesimmetrica{a_2}{0}{b_2} \in E.$$ If we define
$$
V=\{\xi \in \R^{2\times 2}: \lambda_1(\xi)\lambda_2(\xi)=a_2
b_2\}, \,\,Y=V\cap \partial {\textnormal{Pco}E}
$$
we have that $\xi \in \textnormal{rel int}Y$ \footnote{relative
interior of $Y$}. Let $Z$ be the rank one matrix defined by
$$
 Z=\dys\matrice{1}{-\frac{\lambda_2(\xi)}{\lambda_1(\xi)}}{1}
{-\frac{\lambda_2(\xi)}{\lambda_1(\xi)}}:
$$
then
$\lambda_1(\xi+tZ)\lambda_2(\xi+tZ)=\lambda_1(\xi)\lambda_2(\xi)=a_2
b_2 \forall\,\,t \in \R.$ This implies, as $Y$ is compact, that
there exist $t_1<0<t_2: \xi+t_iZ \in
\partial Y, i=1,2.$ Consequently $\xi+t_iZ$ satisfies either (\ref{primarcoe2})
and (\ref{secondarcoe2}) as equalities or (\ref{terzarcoe2}) and
(\ref{secondarcoe2}) as equalities: from the previous studies we
obtain that $\xi+t_iZ \in \textnormal{Rco}E, i=1,2$ and so $\xi
\in \textnormal{Rco}E$.

In the next step we can assume that (\ref{primarcoe2}) and
(\ref{secondarcoe2}) are satisfied as strict inequalities.

{\it{step 3)}} We assume that $\xi$ satisfies (\ref{terzarcoe2})
with equality. Using the explicit expressions of $\lambda_1,
\lambda_2$ (see (\ref{valorisingolari})), it is easy to prove that
if $(\lambda_1(\xi),\lambda_2(\xi))=(x,y)$ the matrix defined by
$$
A=\matrice{1}{\sqrt{\frac{y-\bar{\theta}}{x+\bar{\theta}}}}
{-\sqrt{\frac{y-\bar{\theta}}{x+\bar{\theta}}}}
{-{\frac{y-\bar{\theta}}{x+\bar{\theta}}}}
$$
has the following properties: it is well defined (as $y\geq
\thetabar$ because $\thetabar^2\leq \max\limits_{(a,b)\in K}
ab+\thetabar (b-a)= xy+\thetabar(y-x)$) and
\begin{equation}
\begin{array}{l}
\lambda_1(\xi+tA)\lambda_2(\xi+tA)+\bar{\theta}[\lambda_2(\xi+tA)-\lambda_1(\xi+tA)]=
xy+\bar{\theta}(y-x)
 \vspace{0.2cm}\\
 \forall\,\,t \in [t_{-},t_{+}],\displaystyle
\,\,t_{-}=-\frac{xy(x+\bar{\theta})} {\bar{\theta}(x+y)}\,,\,\,
t_{+}=\frac{(y-x)(x+\bar{\theta})} {x+y}\,.
\end{array}
\label{proprietaA}
\end{equation}In fact
\[
\begin{array}{l}
\lambda_1(\xi+tA)\lambda_2(\xi+tA)=|\det (\xi+tA)|=\dys\left|xy-xt
\frac{y-\thetabar}{x+\thetabar}+ty \right |
\vspace{0.2cm};\\
(\lambda_2(\xi+tA)-\lambda_1(\xi+tA)
)^2=\norma{\xi+tA}^2-2|\det(\xi+tA)|=
\\
\dys =(x+t)^2+\left(y-t\frac{y-\thetabar}{x+\thetabar}\right)^2+
2t^2\frac{y-\thetabar}{x+\thetabar}
-2\dys\left|xy-xt\frac{y-\thetabar}{x+\thetabar}+ty\right| \,.
\end{array}
\]
If we assume that $\dys
xy-xt\frac{y-\thetabar}{x+\thetabar}+ty\geq 0$ (that is $t\ge
t_{-}$) \rife{proprietaA} is equivalent to show that
$$
\thetabar(y-x)+xt\frac{y-\bar{\theta}}{x+\bar{\theta}}-ty=\thetabar
\sqrt{\norma{\xi+tA}^2-2[\det(\xi+tA)]}\,.
$$
If we assume that $\dys
\bar{\theta}(y-x)+xt\frac{y-\bar{\theta}}{x+\bar{\theta}}-yt\geq
0$ (that is $t\leq t_{+}$) we get
\[
\begin{array}{l}\dys
t^2x^2\left(\frac{y-\thetabar}{x+\thetabar}\right)^2+t^2y^2+
2\thetabar(y-x)xt
\frac{y-\thetabar}{x+\thetabar}-2\thetabar(y-x)yt-2t^2xy
\frac{y-\thetabar}{x+\thetabar}=
\\\dys
\thetabar^2\left[
t^2+2xt+t^2\left(\frac{y-\thetabar}{x+\thetabar}\right)^2-2yt
\frac{y-\thetabar}{x+\thetabar}+2t^2\frac{y-\thetabar}{x+\thetabar}
+2xt\frac{y-\thetabar}{x+\thetabar}-2yt \right]\,.
\end{array}
\]
One can easily check that this equality is verified for every $t$
and so (\ref{proprietaA}) is verified.

We prove now that there exists $t_1\in \left[t_{-},0\right]$ such
that $\lambda_2(\xi+t_1A)=b_1:$ this implies that $\xi+t_1 A$
satisfies (\ref{primarcoe2}) and (\ref{terzarcoe2}) as equalities:
as we saw in the first step, $\xi+t_1 A \in \textnormal{Rco}E$.
Moreover we prove also that there exists
 $t_2\in \left[0,t_{+}\right]$ such that
 $\lambda_1(\xi+t_2A)\lambda_2(\xi+t_2A)=a_2b_2:$ this implies that
 $\xi+t_2 A$ satisfies (\ref{secondarcoe2}) and
(\ref{terzarcoe2}) as equalities: as we saw in the second step,
$\xi+t_2 A \in \textnormal{Rco}E$. Consequently $\xi \in
\textnormal{Rco}E$, as it can be written as rank one combination
of $\xi+t_1 A$ and $\xi+t_2 A$.

\noindent {\it{Existence of $t_1$)}} We consider $
F(t)=\lambda_2(\xi+tA)-b_1.$ The existence of $t_1$ follows from
the fact that this function is continuous and $F(0)<0<F(t_{-}):$
in fact
\[
\begin{array}{l}
F(t_{-})=\norma{\xi+t_{-} A}-b_1>0\Longleftrightarrow
\\b_1<\norma{\xi+t_{-} A}=
\vspace{0.2cm}
\\
\displaystyle =
\sqrt{\left(\frac{x^2(\bar{\theta}-y)}{\bar{\theta}(x+y)}
\right)^2+\left(\frac{y^2(\bar{\theta}+x)}
{\bar{\theta}(x+y)}\right)^2+2
\frac{x^2y^2(x+\bar{\theta})(y-\bar{\theta})}
{\bar{\theta}^2(x+y)^2}}\vspace{0.2cm}
\\\displaystyle
=\frac{x^2(y-\thetabar)+y^2(\thetabar +x)}{\bar{\theta}(x+y)}
=\frac{xy+\bar{\theta}(y-x)}{\bar{\theta}}\,.
\end{array}
\]
The last inequality is equivalent to
$$
\bar{\theta}b_1\leq
xy+\bar{\theta}(y-x)=a_1b_1+\bar{\theta}(b_1-a_1)\Longleftrightarrow
b_1\ge \bar{\theta}
$$
which is true.

\noindent {\it{Existence of $t_2$)}} We consider $
G(t)=\lambda_1(\xi+tA)\lambda_2(\xi+tA)-a_2 b_2.$ The existence of
$t_2$ follows from the fact that $G$ is continuous and
$G(0)<0<G(t_{+}):$ in fact $G(t_{+})>0$ if and only if
$$
\left|xy-x\frac{(y-\bar{\theta})(y-x)}{x+y}+
y\frac{(x+\bar{\theta})(y-x)}{x+y}\right| \geq a_2b_2\,.
$$
Using that $xy+\thetabar(y-x)=a_2b_2+\bar{\theta}(b_2-a_2)$ we get
$$xy(x+y)+\bar{\theta}(y-x)(y+x)\geq a_2b_2(x+y)\Longleftrightarrow
a_2b_2+\bar{\theta}(b_2-a_2)\geq a_2b_2
$$
which is true.
\end{proof}
To prove theorem \ref{rcoegise} it will be useful to recall the
following properties about convex functions and their
sub-differential (we will follow the definition of
\cite{rockafellar}).
\begin{defn}
Let $f: \R\to \R$ be a convex function and $\thetabar \in \R$. The
sub-differential of $f$ in $\thetabar$ is the set
$$\partial f(\thetabar)=\{\theta^* \in \R:
f(\theta)\geq
f(\thetabar)+\theta^*(\theta-\thetabar)\quad\forall\,\,\theta \in
\R\}\,.
$$
\end{defn}
\begin{prop}\label{propsubdiff}
Let $f: \R\to \R$ be a convex function. Then

\noindent i) $\partial f(\theta)$ is non empty, compact and convex
for every $\theta \in \R.$

\noindent ii) If $\theta$ is a point of differentiability of $f$
then $\partial f(\theta)=\{f'(\theta)\}\,.$

\noindent iii) The set of points of differentiability of $f$ is
dense in $\R$ and
$$
\partial f(x)=\textnormal{co}\,\overline{S(x)}\,,\,\,
 S(x)=\{\lim\limits_{n\to \infty}f'(x_n)\,\, f
\text{differentiable in}\,\, x_n, x_n\to x\}\,.
$$
\end{prop}
\begin{rem}
The proof of {\it{i)}} can be found at page 218 of
\cite{rockafellar}; {\it{ii)}} is theorem 26.1 of
\cite{rockafellar}; the proof of {\it{iii)}} is a direct
combination of theorems 25.6 and 17.2 of \cite{rockafellar}.
\end{rem}
We pass now to the proof of theorem \ref{rcoegise}.
\begin{proof}
Thanks to propositions \ref{policard} and \ref{prelim} it is
sufficient to prove that $\textnormal{Pco}\,E=
\textnormal{Rco}\,E.$ For this, as in lemma
\ref{numerofinitodipunti}, we will show the inclusion
$$\partial {\textnormal{Pco}\,E}\subseteq \textnormal{Rco}\,E\,.$$
Let $\xi \in
\partial {\textnormal{Pco}\,E}$. We have seen in remark \ref{bord} that there exists
$\overline{\theta}\in [0,\max\limits_{(a,b)\in
K}b]$ such that
$$
\lambda_1(\xi)\lambda_2(\xi)+\overline{\theta}(\lambda_2(\xi)-\lambda_1(\xi))
=m(\thetabar)\,,
$$
and for every $\theta \in [0,\max\limits_{(a,b)\in K}b]$
$$\lambda_1(\xi)\lambda_2(\xi)+{\theta}(\lambda_2(\xi)-\lambda_1(\xi))
\leq m(\theta)\,.$$ We define
\[
F(\theta)=\left\{
\begin{array}{ll}
\max\limits_{(a,b) \in K}ab&\theta\leq 0 \vspace{0.2cm}
\\
m(\theta)&\theta \in [0,\max\limits_{(a,b)\in K}b]\vspace{0.2cm}
\\
\theta^2&\theta\geq \max\limits_{(a,b)\in K}b\,.
\end{array}
 \right.
\]
We observe that $\xi$ satisfies
$$
\lambda_1(\xi)\lambda_2(\xi)+\theta(\lambda_2(\xi)-\lambda_1(\xi))
\leq F(\theta),\quad\forall\,\,\theta \in \R\,,
$$
and there exists $\overline{\theta}\in [0,\max\limits_{(a,b)\in
K}b]$ such that
$$
\lambda_1(\xi)\lambda_2(\xi)+\overline{\theta}(\lambda_2(\xi)-\lambda_1(\xi))
=F(\thetabar)\,.
$$
The following remarks will be useful:

$\bullet$ $\lambda_2(\xi)-\lambda_1(\xi) \in
\partial F(\thetabar)\,.$

$\bullet$ One can easily check that $F$ is convex: the previous
proposition implies that $\partial
F(\thetabar)=[\alpha(\thetabar), \beta(\thetabar)]$ for some
$\alpha(\thetabar), \beta(\thetabar) \in \R$.

$\bullet$ If $\theta_0$ is a point of differentiability for the
function $F$ then $\partial F(\theta_0)=\{F'(\theta_0)\}=
b_0-a_0\,$ for some $(a_0,b_0)\in K$, such that
$F(\theta_0)=\max\limits_{(a,b)\in K}\{ab+\theta_0(b-a)\}=a_0
b_0+\theta_0(b_0-a_0)\,.$

We are going to show that $\xi \in \textnormal{Rco}\,E$: for this
we will find a set $a\subset K$ composed by one or two points such
that letting $A=\{\xi \in \R^{2\times 2}:
(\lambda_1(\xi),\lambda_2(\xi))\in a\}$ we have $\xi \in
{\textnormal{Rco} A}\subseteq \text{Rco}\,E$. We will distinguish
the cases $\thetabar=0$, $\thetabar \in (0,\max\limits_{(a,b)\in
K}b)$, $\thetabar=\max\limits_{(a,b)\in K}b\,$ respectively in the
steps 1, 2, 3.

\noindent {\it{step 1)}} We analyse the case $\thetabar=0$, for
which $\max\limits_{(a,b) \in K}ab=\lambda_1(\xi)\lambda_2(\xi)$.
We study the set $S(0)$ defined in proposition \ref{propsubdiff}
(the points $\theta_n$ will be points of differentiability of $F$
throughout this proof):
$$
S(0)=\{\lim\limits_{n\to \infty}F'(\theta_n), \theta_n\to 0\}=
\{\lim\limits_{n\to \infty}F'(\theta_n), \theta_n\to 0^+\}\cup
\{0\}
$$
as for every $\theta<0$ $F$ is constant. Let
$$p_M=\sup\{\lim\limits_{n\to \infty}F'(\theta_n), \theta_n\to
0^+\}\,\geq 0\,.$$ Let $\theta_n \to 0^+$ be points of
differentiability for $F$: then $F'(\theta_n)=b_n-a_n$, for some
$(a_n,b_n) \in K;$ therefore every point of $S(0)$ is $0$ or $b-a$
for some $(a,b)\in K$, because of the compactness of $K$. The fact
that $K$ is compact implies also that $p_M= \bar{b}- \bar{a},$ for
some $(\bar{a},\bar{b})\in K$ and so
$$
\partial F(0)=\textnormal{co} \overline{S(0)}=[0,\bar b-\bar a]\ni
\lambda_2(\xi)-\lambda_1(\xi)\,.
$$
It is easy to see that there exists $(\tilde{a},\tilde{b}) \in K$
such that $\bar b-\bar a=\tilde b-\tilde a$ and
$\max\limits_{(a,b)\in K}ab=\tilde a\tilde b$. In fact by
definition of $\bar b-\bar a$,  $\forall\,\varepsilon>0$ there
exists $\theta_n^{\varepsilon}$ which goes to $0$ for $n\to
\infty$ and a sequence $(a_n^{\varepsilon},b_n^{\varepsilon}) \in
K$ such that
\begin{equation}
\bar b-\bar a-\varepsilon\leq \lim\limits_{n\to
\infty}F'(\theta_n^{\varepsilon})=\lim\limits_{n\to
\infty}b_n^{\varepsilon}-a_n^{\varepsilon}\leq \bar b-\bar a.
\label{defsupthetabarzero}
\end{equation}
We observe that, as $\theta_n^{\varepsilon}$ is a point of
differentiability of $F$
$$\max\limits_{(a,b)\in
K}ab+\theta_n^{\varepsilon}(b-a)=
a_n^{\varepsilon}b_n^{\varepsilon}+\theta_n^{\varepsilon}
(b_n^{\varepsilon}-a_n^{\varepsilon})\,.
$$
Now, if we consider the points
$(a_n^{\varepsilon},b_n^{\varepsilon}) \in K$, as $K$ is compact,
we can say, up to a sub-sequence that
$(a_n^{\varepsilon},b_n^{\varepsilon})\to
(a^{\varepsilon},b^{\varepsilon}) \in K\,.$ For the same reason,
if $\varepsilon\to 0$ $(a^{\varepsilon},b^{\varepsilon})\to
(\tilde{a},\tilde{b})\in K$\,. Passing to the limit for $n\to
\infty$ in the last relation, we obtain from the continuity in
$\theta$ of the function $\max\limits_{(a,b)\in K}ab+\theta(b-a)$
$$
\lim\limits_{n\to \infty}\max\limits_{(a,b)\in
K}ab+\theta_n^{\varepsilon}(b-a)=\max\limits_{(a,b)\in
K}ab=\lim\limits_{n\to
\infty}a_n^{\varepsilon}b_n^{\varepsilon}+\theta_n^{\varepsilon}
(b_n^{\varepsilon}-a_n^{\varepsilon})=a^{\varepsilon}b^{\varepsilon}\,,
$$
and so
$$
\max\limits_{(a,b)\in
K}ab=\lambda_1(\xi)\lambda_2(\xi)=\lim\limits_{\varepsilon\to
0}a^{\varepsilon}b^{\varepsilon}=\tilde{a}\tilde{b}\,.
$$
From the relation (\ref{defsupthetabarzero}) we get,
$$
\bar b-\bar a\leq \lim\limits_{\varepsilon\to 0 }\lim\limits_{n\to
\infty}b_n^{\varepsilon}-a_n^{\varepsilon}=\tilde{b}-\tilde{a}\leq
\bar b-\bar a\Longleftrightarrow \bar b-\bar
a=\tilde{b}-\tilde{a}\,.
$$

Then we have that
$\lambda_1(\xi)\lambda_2(\xi)=\tilde{a}\tilde{b}$ and
$\lambda_2(\xi)-\lambda_1(\xi)\leq \tilde{b}-\tilde{a}$, that is
$\lambda_2(\xi)\leq \tilde{b}\,.$ This is equivalent, thanks to
remark \ref{onepoint}, to say that $\xi \in
\textnormal{Rco}A\subseteq \textnormal{Rco}E$, where
$$A=\{\xi \in \R^{2\times 2}:
(\lambda_1(\xi),\lambda_2(\xi))=(\tilde{a},\tilde{b})\}\,.$$

\noindent {\it{step 2)}} We study the case $\thetabar \in
(0,\max\limits_{(a,b)\in K}b)$. As in step 1, if
\[
\begin{array}{l}
p_m=\inf\{\lim\limits_{n\to \infty}F'(\theta_n), \theta_n\to
\thetabar\}\\
p_M=\sup\{\lim\limits_{n\to \infty}F'(\theta_n), \theta_n\to
\thetabar\}\,,
\end{array}
\]
we have for some $(a_i,b_i) \in K, i=1,2$
$$
\partial F(\thetabar)=[b_2-a_2,b_1-a_1]\,.
$$
Following the same kind of study as in step 1, one can show that
there exist $(\tilde a_i,\tilde b_i)\in K, i=1,2$ such that
\[
\begin{array}{l}
\lambda_1(\xi)\lambda_2(\xi)+\thetabar(\lambda_2(\xi)-\lambda_1(\xi))=F(\thetabar)
=\max \limits_{(a,b)\in K}\{ab+\thetabar(b-a)\}=
\\
= \tilde{a_i}\tilde{b_i}+\bar{\theta}(\tilde{b_i}-\tilde{a_i}),
\,\,\,\,\,\,\,\,\,\,\tilde{b_i}-\tilde{a_i}=b_i-a_i.
\end{array}
\]
 We now show that $\xi\in
\textnormal{Rco}\,A$, where $$A=\{\xi \in \R^{2\times 2}:
 (\lambda_1(\xi),\lambda_2(\xi))=(\tilde{a_i},\tilde{b_i}),\,i=1,2\}\,:$$
 for this, thanks to lemma \ref{numerofinitodipunti}, it is sufficient to show that
\[
\begin{array}{l}
\tilde{b_2}\leq  \lambda_2(\xi)\leq \tilde{b_1},\vspace{0.1cm}\\
\tilde{a_1}\tilde{b_1}\leq  \lambda_1(\xi)\lambda_2(\xi)\leq
\tilde{a_2}\tilde{b_2}, \vspace{0.1cm}\\
\tilde{a_1}<  \tilde{a_2} \,.
\end{array}
\]
For every $\theta \in [0,\max\limits_{(a,b)\in K}b]$ we have that
$\theta^2\leq \max\limits_{(a,b)\in K}ab+\theta(b-a)$, as seen in
proposition \ref{prelim}: writing this inequality for
$\theta=\thetabar$ we get $\lambda_2(\xi)>\thetabar$. As
$$
\tilde{a_2}\tilde{b_2}+\theta(\tilde{b_2}-\tilde{a_2})<
\lambda_1(\xi)\lambda_2(\xi)+\theta(\lambda_2(\xi)-\lambda_1(\xi))<
\tilde{a_1}\tilde{b_1}+\theta(\tilde{b_1}-\tilde{a_1})
$$
for every $\theta\geq \thetabar$ and in particular for
$\theta=\lambda_2(\xi)$ we have that $\tilde{b_2}\leq
\lambda_2(\xi)\leq \tilde{b_1}$, that is the first condition is
verified. As $\tilde{b_2}-\tilde{a_2}\leq
\lambda_2(\xi)-\lambda_1(\xi)\leq \tilde{b_1}-\tilde{a_1}$, then
$\tilde{a_1}\tilde{b_1}\leq \lambda_1(\xi)\lambda_2(\xi)\leq
\tilde{a_2}\tilde{b_2}$ and $\tilde{a_1}<  \tilde{a_2}\,.$

\noindent {\it{step 3)}} We study the case
$\thetabar=\max\limits_{(a,b)\in K}b$: we have, as we saw in
proposition \ref{prelim} $\lambda_2(\xi)=\max\limits_{(a,b)\in
K}b.$ We define
$$p_m=\inf\{\lim\limits_{n\to \infty}F'(\theta_n), \theta_n\to
\max\limits_{(a,b)\in K}b\}\,$$ and we have, as in the previous
steps
$$
\partial F(\thetabar)=[\bar b-\bar a, \beta(\thetabar)]\,,(\bar a,\bar
b)\in K\,.
$$
It is easy to show that there exists $(\tilde{a},\tilde{b})\in K$
such that $\bar b-\bar a=\tilde{b}-\tilde{a}\,$ and
$\max\limits_{(a,b)\in K
}ab+\bar{\theta}(b-a)=[\max\limits_{(a,b)\in
K}b]^2=\tilde{a}\tilde{b}+\thetabar(\tilde{b}-\tilde{a})$.
Therefore $\tilde{b}=\max\limits_{(a,b)\in K}b\,.$ Defining
$$A=\{\xi \in \R^{2\times 2}:
(\lambda_1(\xi),\lambda_2(\xi))=(\tilde{a},\tilde{b})\}$$ we have
that $\xi \in \textnormal{Rco}\,A \subseteq \textnormal{Rco}\,E$,
thank to remark \ref{onepoint}: in fact
$\lambda_2(\xi)-\lambda_1(\xi)\geq
\tilde{b}-\tilde{a}=\max\limits_{(a,b)\in K}b-\tilde{a}\,,$ that
is $\lambda_1(\xi)\leq \tilde{a}\,,$ and
$\lambda_2(\xi)=\tilde{b}\,.$

The formula for the interior of $\text{Rco}\,E$ follows from
proposition \ref{interno} and from the fact that
$\text{Pco}\,E=\text{Rco}\,E\,,$ as we have just showed.
\end{proof}

\section{The existence theorem}
In this section we are going to show theorem \ref{teorema
d'esistenza nel nostro caso dacorognamarcellini}. The proof will
be a direct combination of theorem \ref{Giovio} and of proposition
\ref{approximation}. To do this it will be useful to define the so
called approximation property (this definition is given in
\cite{implicitdacomarc}).
\begin{defn}
Let $E\subset K(E)\subset \R^{2\times 2}$. We say that $E$ and
$K(E)$ have the approximation property if there exists a family of
closed sets $E_\delta$ and $K(E_\delta)$, $\delta>0$, such that

\noindent 1) $E_\delta\subset K(E_\delta)\subset\textnormal{int}
K(E)$ for every $\delta>0;$

\noindent 2) for every $\varepsilon>0$ there exists
$\delta_0=\delta_0(\varepsilon)>0$ such that
\textnormal{dist}$(\eta,E)\leq \varepsilon$ for every $\eta \in
E_\delta$ and $\delta \in [0,\delta_0];$

\noindent 3) if $\eta \in \textnormal{int} K(E)$ then $\eta \in
K(E_\delta)$ for every $\delta>0$ sufficiently small.
\end{defn}
We can now show the following result.
\begin{prop}\label{approximation}
Let $E$ be defined by
$$
E=\{\xi \in \R^{2\times 2}: (\lambda_1(\xi),\lambda_2(\xi))\in K\}
$$
with $K$ compact satisfying
(\ref{ipotesidipositivitadeipuntidiK}).
 Then
 $E$ and $\textnormal{Rco}E$ have the approximation property with
 $K(E_{\delta})=\textnormal{Rco}\,E_{\delta},$ if $$
 E_{\delta}=\bigcup\limits_{(a,b)\in K}E_{\delta}^{(a,b)},\,\,\,\, \,
E_{\delta}^{(a,b)}=\{\xi \in \R^{2\times 2}: (\lambda_1(\xi),
\lambda_2(\xi))=(a-\delta,b-\delta)\},$$ for $0\leq\delta\leq
\min\limits_{(a,b)\in K}a/2$.
\end{prop}
It will be useful the following result due to Cardaliaguet and
Tahraoui \cite{card-tah}:
\begin{prop}\label{propfunzioneHtheta}
For every $\theta\geq 0$ the function $H_{\theta}:\R^{2\times
2}\to \R$ defined by
$$
H_{\theta}(\xi)=\max\{\lambda_1(\xi)\lambda_2(\xi)+
\theta(\lambda_2(\xi)-\lambda_1(\xi))- \theta^2, 0\}
$$
is rank one convex.
\end{prop}
\begin{rem}
In \cite{card-tah} Cardaliaguet and Tahraoui show that the
function $H_{\theta}$ is polyconvex.
\end{rem} We can now prove proposition \ref{approximation}.
\begin{proof}
We remark that $E_\delta$ is compact: this will let us use the
representation theorem \ref{rcoegise}. In the following three
steps, we show the three conditions of the approximation property
respectively.

\noindent {{$1)  \mbox{Rco}(E_\delta)\subset\mbox{intRco}E,
\forall\,\, \delta>0:$}}

Let $(a,b)\in K$ be fixed. Then we have
\[
\begin{array}{l}
E_{\delta}^{(a,b)} \subseteq \{\xi \in \R^{2\times 2}:
\lambda_1(\xi)\lambda_2(\xi)<ab, \,\lambda_2(\xi)<b\}
\vspace{0.2cm}
\\
\phantom{E_{\delta}^{(a,b)} }=\textnormal{int Rco}\{\xi \in
\R^{2\times 2}: (\lambda_1(\xi),
\lambda_2(\xi))=(a,b)\}\vspace{0.2cm}
\\\phantom{E_{\delta}^{(a,b)} }
\subseteq \textnormal{int Rco} E;
\end{array}
\]
this implies, if we pass to the union over $K$, that
$$
E_{\delta}\subseteq \textnormal{int Rco} E\,.
$$
From this inclusion we can infer that
$$
\textnormal{Rco}E_{\delta}\subseteq
\textnormal{Rco}(\textnormal{int Rco} E)=\textnormal{int Rco} E\,,
$$
as the interior of $\textnormal{Rco}\,E$ is rank one convex. In
fact let $\xi, \xi+A \in \textnormal{int\,Rco}E,$ with $rk(A)=1,$
that is, thank to theorem \ref{rcoegise}, for every $\theta \in
[0,\max\limits_{(a,b)\in K }b]$
\begin{equation}
\begin{array}{l}
\lambda_1(\xi)\lambda_2(\xi)+
\theta(\lambda_2(\xi)-\lambda_1(\xi))< \max\limits_{(a,b)\in
K}ab+\theta(b-a)
\\
\lambda_1(\xi+A)\lambda_2(\xi+A)+
\theta(\lambda_2(\xi+A)-\lambda_1(\xi+A))< \max\limits_{(a,b)\in
K}ab+\theta(b-a)\,.
\end{array}
\label{disinternoproprappross}
\end{equation}
We want to show that $$\xi+sA\in \textnormal{int\,Rco}E, s\in
[0,1].$$ Surely $\xi+sA\in \textnormal{Rco}E,$ because $\xi, \xi +
A \in \textnormal{Rco}E.$ Now, let us suppose that there exists
$\thetabar \in [0,\max\limits_{(a,b)\in K}b]$ such that
\begin{equation}
\max\limits_{(a,b)\in
K}ab+\thetabar(b-a)=\lambda_1(\xi+sA)\lambda_2(\xi+sA)+
\thetabar(\lambda_2(\xi+sA)-\lambda_1(\xi+sA))\,.
\label{thetabarapprossimazione} \end{equation} We can assume that
$\thetabar \neq \max\limits_{(a,b)\in K}b$. In fact, due to
(\ref{uguale})
$$[\max\limits_{(a,b)\in K}b]^2=\max\limits_{(a,b)\in
K}\{ab+[\max\limits_{(a,b)\in K}b](b-a)\}\,;$$ therefore if we
choose $\theta=\max\limits_{(a,b)\in K}b$ in
\rife{disinternoproprappross} and in
\rife{thetabarapprossimazione} we have
$\lambda_2(\xi),\lambda_2(\xi+A)<\max\limits_{(a,b)\in K}b\,$ and
$\lambda_2(\xi+sA)=\max\limits_{(a,b)\in K}b\,.$ This is a
contradiction as $\lambda_2$ is a norm over $\R^{2\times 2}$.
Therefore we can write
$$
\thetabar^2<\max\limits_{(a,b)\in
K}ab+\thetabar(b-a)=\lambda_1(\xi+sA)\lambda_2(\xi+sA)+
\thetabar(\lambda_2(\xi+sA)-\lambda_1(\xi+sA))\,.
$$
Using the expression of  the function $H_{\thetabar}$ defined in
proposition \ref{propfunzioneHtheta}, we  have
\[
\begin{array}{l}
H_{\thetabar}(\xi+sA)=
\lambda_1(\xi+sA)\lambda_2(\xi+sA)+\thetabar(\lambda_2(\xi+sA)-\lambda_1(\xi+sA))-\thetabar^2
\\\phantom{H_{\thetabar}(\xi+sA)}=\max\limits_{(a,b)\in K}ab+\thetabar(b-a)-\thetabar^2>0.
\end{array}
\] Thanks
to the fact that $H_{\thetabar}(\xi)$ is rank one convex, from
proposition \ref{propfunzioneHtheta}
\[
0<H_{\thetabar}(\xi+sA)\leq
sH_{\thetabar}(\xi+A)+(1-s)H_{\thetabar}(\xi)
\leq\max\{H_{\thetabar}(\xi),H_{\thetabar}(\xi+A)\}\,.
\]
Without loss of generality we can assume that
$\max\{H_{\thetabar}(\xi),H_{\thetabar}(\xi+A)\}=H_{\thetabar}(\xi+A)\,.$
If $H_{\thetabar}(\xi+A)=0$ we have a contradiction. If
$H_{\thetabar}(\xi+A)>0,$ we have, as $\xi+A \in
\textnormal{int\,Rco}E$
\[
\begin{array}{l}
H_{\thetabar}(\xi+A)=\lambda_1(\xi+A)\lambda_2(\xi+A)+
\thetabar(\lambda_2(\xi+A)-\lambda_1(\xi+A))-\thetabar^2\\\vspace{0.1cm}
< \max\limits_{(a,b)\in K}ab+\thetabar(b-a)-\thetabar^2
\end{array}
\]
and so we have obtained
$$
H_{\thetabar}(\xi+sA)=\max\limits_{(a,b)\in
K}ab+\thetabar(b-a)-\thetabar^2\leq
H_{\thetabar}(\xi+A)<\max\limits_{(a,b)\in
K}ab+\thetabar(b-a)-\thetabar^2
$$
which is a contradiction.

\noindent {{$2) \forall \,\,\varepsilon>0\, \exists\,
\delta_0=\delta_0(\varepsilon)>0: \mbox{dist}(\eta,E)\leq
\varepsilon \,\,\forall\,\,\, \eta \in E_{\delta},  \delta \in
[0,\delta_0]:$}}

Let $\eta\in E_{\delta}$; then there exists $(a,b)\in K$ such that
$ \eta \in E_{\delta}^{(a,b)}$. We define
$$
X=\matricesimmetrica{a}{0}{b}\in E\,.
$$
Let $A,B \in \mathcal{O}(2)$ be such that $A \eta
B=\matricesimmetrica{\lambda_1(\eta)}{0}{\lambda_2(\eta)}$. Then
$$
\normasemplice{\eta-A^{-1}X B^{-1}}=\normasemplice{A \eta B-A
A^{-1}X B^{-1} B}= \normasemplice{A \eta B-X}= \sqrt{2
\delta^2}\,.
$$
This implies that
$$
\textnormal{dist}(\eta, E)\leq
\normasemplice{\eta-A^{-1}XB^{-1}}=\sqrt{2 \delta^2}\to 0, \delta
\to 0;
$$
moreover this limit is uniform with respect to $\eta$.

\noindent {{$3) \mbox{If} \,\,\eta \,\in \mbox{intRco}E
\,\,\mbox{then}\,\, \eta \,\in \,\mbox{Rco}E_\delta \,\,
\forall\,\, \delta>0 \,\,\mbox{sufficiently small:} $}}

\noindent Let $\eta \in \mbox{intRco}E$; if
$(\lambda_1(\eta),\lambda_2(\eta))=(x,y)$ thanks to theorem
\ref{rcoegise} we have to show the following implication:
\[
\begin{array}{c}
xy+\theta(y-x)<\max\limits_{(a,b)\in K}ab+\theta(b-a)
\,\,\forall\,\theta \in [0,\max\limits_{(a,b)\in K}b]
\\
\Downarrow\vspace{0.2cm}
\\
xy+\theta(y-x)<\max\limits_{(a,b)\in
K}(a-\delta)(b-\delta)+\theta(b-a)\,
\end{array}
\] uniformly with respect to $\theta \in
[0,\max\limits_{(a,b)\in K}b-\delta]$. For this it is sufficient
to show that
$$
\lim\limits_{\delta\to 0}\max\limits_{(a,b)\in
K}(a-\delta)(b-\delta)+\theta(b-a)=\max\limits_{(a,b)\in
K}ab+\theta(b-a)
$$
uniformly with respect to $\theta \in [0,\max\limits_{(a,b)\in
K}b-\delta]$. We have, as $(a-\delta)(b-\delta)+\theta(b-a)\geq
0\quad\forall\,\,\delta \in [0,\min\limits_{(a,b)\in K}a/2]$,
\[
\begin{array}{l}
|\max\limits_{(a,b)\in
K}(a-\delta)(b-\delta)+\theta(b-a)-\max\limits_{(a,b)\in
K}ab+\theta(b-a)|\vspace{0.2cm}
\\
 \leq \max\limits_{(a,b)\in
K}|(a-\delta)(b-\delta)+\theta(b-a)-ab-\theta(b-a)| \vspace{0.3cm}
\\
 \leq \max\limits_{(a,b)\in
K}\delta(a+b+\delta)\to 0, \,\delta\to 0\,,
\end{array}
\]
uniformly with respect to $\theta \in [0,\max\limits_{(a,b)\in
K}b-\delta]$. Consequently we showed the third condition of the
approximation property too.
\end{proof}
\subsection{Proof of the existence theorem}
 We are going to recall an abstract existence theorem
(established by Dacorogna and Pisante \cite{Giovio}) that we will
apply.
\begin{thm}\label{Giovio}
Let $\Omega\subset\mathbb{R}^{2}$ be an open set. Let
$E\subset\mathbb{R}^{2\times 2}$ be a compact set. Assume that $E$
and $\textnormal{Rco}\,E$ have the approximation property with
$K(E_{\delta})=\textnormal{Rco}\,E_{\delta}$. Let $\varphi\in
C_{piec}^{1}\left( \overline{\Omega};\mathbb{R}^{2}\right)$ be
such that
\[
D\varphi\left(  x\right)  \in E\cup\operatorname*{int\,Rco}E\text{, a.e. in }%
\Omega.
\]
Then there exists $u\in\varphi+W_{0}^{1,\infty}\left(
\Omega;\mathbb{R}^{2}\right) $ such that
\[
Du\left(  x\right)  \in E\text{, a.e. in }\Omega.
\]
\end{thm}
Our theorem then follows immediately. In fact, $E$ is compact and
we verified that $E$ and Rco$E$ have the approximation property in
proposition \ref{approximation}. Then we obtain the existence
theorem \ref{teorema d'esistenza nel nostro caso
dacorognamarcellini} thanks to theorem \ref{Giovio}.
\section{Representation of ${\text{Rco}}\,E$}
In this section we are going to give an explicit formula of
Rco$\,E$ for a set $E$ (\ref{notreE}) defined by a set $K$
composed by a finite number of points. We will treat the cases in
which $K$ is composed by one, two elements, and finally a
particular $K$ composed by three elements; for the general formula
and for its proof we refer to \cite{gisedottoratotesi}. We recall
that the representation of ${\text{Rco}}\,E$, for $E$ defined by
$K$ composed by one element was already obtained by Dacorogna and
Marcellini (see \cite{implicitdacomarc}); Cardaliaguet and Tahraoui
in \cite{card-tah} showed the formula for the case of $K$ composed
by two elements.

To give the representation of ${\text{Rco}}\,E$ it is sufficient
to give the formula of the function $\sigma$ defined by
(\ref{sigmacard-ta}): indeed ${\text{Rco}}\,E=\{\xi \in
\R^{2\times 2}: \lambda_2(\xi)\leq \sigma(\lambda_1(\xi))\}\,.$
\begin{prop}
Let $E$ be defined by (\ref{notreE}) with $K=\{(a,b), 0<a\leq
b\}.$ Then
$$
\sigma(x)=\inf\left\{\frac{ab}{x},b\right\}\,.
$$
\end{prop}

\begin{prop}
Let $E$ be defined by (\ref{notreE}) with
$K=\{(a_1,b_1),(a_2,b_2), 0<a_i\leq b_i, i=1,2\}.$ Without loss of
generality we can assume that $a_1\geq a_2$. Then

\noindent 1) If $a_1b_1>a_2b_2$ and $b_2>b_1$ then
$$
\sigma(x)=\inf\left\{\frac{a_1b_1}{x},b_2,
\frac{\theta(1,2)x+a_1b_1+\theta(1,2)(b_1-a_1)}{x+\theta(1,2)}\right\}\,
$$
where $\displaystyle
\theta(1,2)=\frac{a_1b_1-a_2b_2}{b_2-a_2-(b_1-a_1)}$.

\noindent 2) If $b_1\geq b_2$ then
$$
\sigma(x)=\inf\left\{\frac{a_1b_1}{x}, b_1 \right\}\,.
$$

\noindent 3) If $b_1<b_2$ and $a_1b_1\leq a_2b_2$, then
$$
\sigma(x)=\inf\left\{\frac{a_2b_2}{x}, b_2 \right\}\,.
$$
\end{prop}
\begin{prop}
Let $E$ be defined by (\ref{notreE}) with $K=\{(a_i,b_i),
0<a_i\leq b_i, i=1,2,3\}.$ We define
$$
\theta(i,j)=\frac{a_ib_i-a_jb_j}{b_j-a_j-(b_i-a_i)}.
$$
Let us suppose that $a_1b_1>a_2b_2>a_3b_3,$ $b_3>\max\{b_1,b_2\},$
$b_2-a_2> b_1-a_1$ and $\theta(1,2)< \theta(1,3)\,.$ Then
$$
\sigma(x)=\inf\left\{\frac{a_1b_1}{x},
b_3,\frac{x\theta(j,j+1)+a_jb_j+\theta(j,j+1)(b_j-a_j)}{x+\theta(j,j+1)},
j=1,2 \right\}\,.
$$
\end{prop}
\vspace{0.2cm} \textit{Acknowledgements.} I would like to thank
Bernard Dacorogna for having proposed me the subject of this
article and for his scientific help; I'm very grateful to the
referees for their precious remarks and suggestions. As well, I
beneficiated of some discussions with Ana Margarida Fernandes
Ribeiro and Giovanni Pisante.
\bibliographystyle{plain}
\bibliography{Croce_isotropic}

\end{document}